\documentclass[reqno,tbtags,a4paper,12pt]{amsproc}

\usepackage[T2A]{fontenc}     
\usepackage[utf8]{inputenc} 
\usepackage[in]{fullpage}

\usepackage{amsmath,amssymb,amsthm,amsfonts,amscd}

\newtheorem*{theorem*}{Theorem}
\newtheorem{lemma}{Lemma}
\newtheorem*{cor*}{Corollary}
\theoremstyle{definition}
\newtheorem*{defin*}{Definition}

\newcommand{\Z}{\mathbb{Z}}

\newcommand{\ord}{\mathop{\mathrm{ord}}\nolimits}
\newcommand{\Sym}{\mathop{\mathrm{Sym}}\nolimits}

\author{Alexander A. Gaifullin}

\thanks{This work was supported by the Russian Science Foundation under grant no. 23-11-00143, \texttt{https:/\!/rscf.ru/project/23-11-00143/} }

\title{Commutativity of involutive two-valued groups}
\date{}

\address{Steklov Mathematical Institute of Russian Academy of Sciences, Moscow, Russia}
\address{Skolkovo Institute of Science and Technology, Moscow, Russia}
\address{Lomonosov Moscow State University, Moscow, Russia}
\address{Institute for the Information Transmission Problems of the Russian Academy of Sciences (Kharkevich Institute), Moscow, Russia}
\email{agaif@mi-ras.ru}

\begin{document}
\maketitle

Over the past decades, theory of $n$-valued groups has been enriched with numerous new results and applications, see~\cite{Buc06,BV19}. One of its features is the complexity of classification problems even in the commutative case. In~\cite{BV19}, Buchstaber and Veselov, in connection with the Conway topograph, introduced a class of \textit{involutive two-valued groups} for which the classification problems seem amenable, see~\cite{BVG22}. Following~\cite{BV19,BVG22}, we give the following definition.

\begin{defin*}
 An \textit{involutive two-valued group} is a set~$X$ endowed with a multiplication $*\colon X\times X\to\Sym^2(X)$ (where $\Sym^2(X)$ is the symmetric square of~$X$) and an identity element $e\in X$ that satisfy the following properties:

 (1) (\textit{associativity}) for any $x,y,z\in X$, there is an equality of $4$-element multisets $(x*y)*z=x*(y*z)$,

 (2) (\textit{strong identity}) $x*e=e*x=[x,x]$ for all $x\in X$,

 (3) (\textit{involutivity}) the multiset $x*y$ contains the identity~$e$ if and only if $x=y$.

 A two-valued group is said to be \textit{commutative} if it in addition satisfies  the following property:

 (4) (\textit{commutativity}) $x*y=y*x$ for all $x,y\in X$.
\end{defin*}

Note that in~\cite{BVEP} the concept of an involutive $n$-valued group is used in another, not equivalent meaning, see Remark~1.5 in~\cite{BVG22} for details.

In~\cite{BVG22}, a complete classification of finitely generated commutative involutive two-valued groups and partial classification results in the non-finitely generated and topological cases were obtained. The main result of the present note is the following theorem, which implies that all classification results of the paper~\cite{BVG22} are valid without the assumption of commutativity.

\begin{theorem*}
 Any involutive two-valued group is commutative.
\end{theorem*}

For single-generated two-valued groups this theorem was proved in~\cite{BVG22}. The following lemma in the commutative case turns into Lemma~2.2 in~\cite{BVG22}; its proof in the general case remains exactly the same.

\begin{lemma}\label{lem_1}
 Suppose that $x$, $y$, and~$z$ are elements of an involutive two-valued group~$X$. Then $z\in x*y$ if and only if $y\in z*x$.
\end{lemma}

Of key importance for us will be the concepts of \textit{powers} and \textit{order} of an element of an involutive two-valued group, which were introduced in~\cite{BVG22}. Namely, for each element $x\in X$ there is a unique sequence $x^0=e,\,x^1=x,\,x^2,\,x^3,\,\ldots$ of elements of~$X$ such that  $x^k*x^m=\bigl[x^{|k-m|},x^{k+m}\bigr]$ for all~$k$ and~$m$. The \textit{order} of~$x$ (denoted by~$\ord x$) is the smallest positive integer~$k$ such that $x^k=e$. In particular, the order of a non-identity element~$x$ is equal to~$2$ if and only if $x*x=[e,e]$.

\begin{lemma}\label{lem_2}
Suppose that $x$ and~$y$ are elements of an involutive two-valued group~$X$ such that $\ord x =2$. Then $x*y=y*x=[z,z]$ for certain element $z\in X$.
\end{lemma}

\begin{proof}
Let $z$ be an element of~$x*y$. By Lemma~\ref{lem_1} the multiset~$z*x$ contains~$y$, that is, $z*x=[y,y']$ for certain~$y'$. Then
$
[y*x,y'*x]=(z*x)*x=z*(x*x)=z*[e,e]=[z,z,z,z].
$
Hence $y*x=[z,z]$. Applying Lemma~\ref{lem_1} again, we obtain that $x*z=[y,y'']$ for certain~$y''$. Then
$
[x*y,x*y'']=x*(x*z)=(x*x)*z=[e,e]*z=[z,z,z,z].
$
Therefore $x*y=[z,z]$.
\end{proof}

\begin{proof}[Proof of Theorem]
Let $x$ and~$y$ be elements of an involutive two-valued group~$X$. We need to prove that  $x*y=y*x$. If one of the two elements~$x$ and~$y$ has order~$2$, the required equality follows from Lemma~\ref{lem_2}. So we may assume that $\ord x>2$ and $\ord y>2$, that is,  $x^2\ne e$ and $y^2\ne e$.
Put $x*y=[z_1,z_2]$ and $y*x=[w_1,w_2]$.
We have the following chain of equalities of $8$-element multisets:
 \begin{multline}\label{eq_main}
  [z_1*w_1,z_1*w_2,z_2*w_1,z_2*w_2]=(x*y)*(y*x)\\{}=x*(y*y)*x=x*\bigl[e,y^2\bigr]*x=\bigl[e,e,x^2,x^2,x*y^2*x\bigr].
 \end{multline}
 Hence either at least two  of the four multisets~$z_i*w_j$ contain~$e$ or one of these four multisets is equal to~$[e,e]$. From the involutivity it follows that  $e\in z_i*w_j$ if and only if $z_i=w_j$. Up to permutations $z_1\leftrightarrow z_2$ and $w_1\leftrightarrow w_2$ and the reversal of the roles of the elements~$x$ and~$y$, there are three substantially different cases.
\smallskip

\textsl{Case 1: $z_1=w_1$ and $z_2=w_2$.} Then we arrive at the required equality~$x*y=y*x$.
\smallskip

\textsl{Case 2: $z_1=z_2=w_1\ne w_2$.} We have
\begin{align}
\label{eq_xyz}
x*(y*z_1)&=(x*y)*z_1=[z_1,z_1]*z_1=\bigl[e,e,z_1^2,z_1^2\bigr],\\
\label{eq_zxy}
(z_1*x)*y&=z_1*(x*y)=z_1*[z_1,z_1]=\bigl[e,e,z_1^2,z_1^2\bigr].
\end{align}
Since $\ord x>2$, we have that $x*x=[e,x^2]$, where $x^2\ne e$. Moreover, by the involutivity $e\notin x*x'$ unless $x'=x$. Hence, from~\eqref{eq_xyz} it follows that $y*z_1=[x,x]$ and $x^2=z_1^2$. Similarly, from~\eqref{eq_zxy} it follows that $z_1*x=[y,y]$ and $y^2=z_1^2$. Therefore $x^2=y^2$. Consequently,
\begin{gather}
 z_1*w_1=z_2*w_1=z_1*z_1=\bigl[e,z_1^2\bigr]=\bigl[e,x^2\bigr],\nonumber\\
 \label{eq_vyklad}
 x*y^2*x=x*x^2*x=\bigl[x,x^3\bigr]*x=\bigl[e,x^2,x^2,x^4\bigr].
\end{gather}
Thus, \eqref{eq_main} reads as
\begin{equation*}
\bigl[e,e,x^2,x^2,z_1*w_2,z_1*w_2\bigr]=\bigl[e,e,e,x^2,x^2,x^2,x^2,x^4\bigr].
\end{equation*}
Since $x^2\ne e$, we obtain that $e\in z_1*w_2$, which leads to a contradiction, since $z_1\ne w_2$. Therefore, the  case under consideration is impossible.
\smallskip

\textsl{Case 3: $z_1*w_1=[e,e]$, that is, $z_1=w_1$ and $\ord z_1=2$.} We may assume that the elements~$z_1$, $z_2$, and~$w_2$ are pairwise different, since otherwise we get into one of the two already considered Cases~1 and~2. By Lemma~\ref{lem_1} we obtain that $y\in z_1*x$. Since $\ord z_1=2$, from Lemma~\ref{lem_2} it follows that $x*z_1=z_1*x=[y,y]$. We have
\begin{gather*}
 (x*z_1)*(z_1*x)=[y,y]*[y,y]=\bigl[e,e,e,e,y^2,y^2,y^2,y^2\bigr],\\
 x*(z_1*z_1)*x=x*[e,e]*x=\bigl[e,e,e,e,x^2,x^2,x^2,x^2\bigr]
\end{gather*}
and hence $x^2=y^2$. As in Case~2, from~\eqref{eq_vyklad} it follows that the $e$ enters the multiset~\eqref{eq_main} with multiplicity at least three. Hence, $e$ is contained in at least one of the three multisets~$z_1*w_2$, $z_2*w_1=z_2*z_1$, and $z_2*w_2$, which is impossible, since the elements~$z_1$, $z_2$, and~$w_2$ are pairwise different. The obtained contradiction completes the proof of the theorem.
\end{proof}

The author is grateful to V.~M.~Buchstaber and A.~P.~Veselov for multiple fruitful discussions.

\end{document}